\documentclass[12pt]{amsart}
\usepackage[utf8]{inputenc}
\usepackage{amsbsy}
\usepackage{amscd}
\usepackage{amsfonts}
\usepackage{amsgen}
\usepackage{amsmath}
\usepackage{amsopn}
\usepackage{amssymb}
\usepackage{amstext}
\usepackage{amsthm}
\usepackage{amsxtra}
\usepackage{array}
\usepackage[all]{xy}
\usepackage{epsfig}
\usepackage{extarrows}
\usepackage{float}
\usepackage{hhline}
\usepackage{longtable}
\usepackage{mathtools}
\usepackage{paralist}
\usepackage{rotating}
\usepackage[hyphens]{url}
\usepackage{hyperref}
\usepackage[hyphenbreaks]{breakurl}
\usepackage{algorithm}
\usepackage[noend]{algpseudocode}
\usepackage[binary-units=true]{siunitx}

\usepackage{graphicx}

\usepackage{tikz}
\usepackage{tikz-cd}
\usetikzlibrary{automata, positioning, arrows}
\tikzset{
    ->,
    >=stealth',
    every state/.style={thick, fill=gray!10},
    initial text=$ $,
}

\usepackage{cleveref}

\theoremstyle{plain}
\newtheorem{theorem}{Theorem}[section]
\newtheorem{proposition}[theorem]{Proposition}
\newtheorem{lemma}[theorem]{Lemma}
\newtheorem{corollary}[theorem]{Corollary}
\newtheorem*{theorem*}{Theorem}
\newtheorem*{proposition*}{Proposition}
\newtheorem*{lemma*}{Lemma}
\newtheorem*{corollary*}{Corollary}
\newtheorem*{theorem*conjecture*}{Theorem/Conjecture}
\theoremstyle{definition}
\newtheorem{definition}[theorem]{Definition}

\newtheorem{remark}[theorem]{Remark}
\newtheorem{example}[theorem]{Example}
\newtheorem{question}[theorem]{Question}

\newtheorem*{definition*}{Definition}
\newtheorem*{conjecture*}{Conjecture}
\newtheorem*{notation*}{Notation}
\newtheorem*{question*}{Question}

\numberwithin{equation}{section}

\newcommand{\N}{\mathbb{N}}
\newcommand{\Z}{\mathbb{Z}}

\newcommand{\C}{\mathbb{C}}
\newcommand{\A}{\mathcal{A}}
\newcommand{\B}{\mathcal{B}}

\newcommand{\OO}{\mathcal{O}}

\newcommand{\abs}[1]{\left\lvert#1\right\rvert}

\newcommand{\Span}[1]{\left\langle#1\right\rangle}
\newcommand{\npoly}[1]{\C\left\langle#1\right\rangle}
\newcommand{\cleene}[1]{{#1}^\star}
\newcommand{\mon}[1]{{#1}^\star}
\newcommand{\textvar}[1]{\mathit{#1}}

\newcommand{\operp}{
    \;\raisebox{0.04cm}
    {\scalebox{0.72}{$\bigcirc\textnormal{\hspace{-14pt}}\perp$}}\;
}

\newcommand{\smalloperp}{
    \raisebox{0.0288cm}
    {\scalebox{0.5184}{$\bigcirc\textnormal{\hspace{-14pt}}\perp$}} \mskip 1.5mu
}

\DeclareMathOperator{\LT}{LT}

\DeclareMathOperator{\Spec}{Spec}

\newcommand{\MagicUnitary}{\begin{pmatrix}
    p\otimes p\otimes p &
    p\otimes (1-p)\otimes q &
    p\otimes p\otimes (1-p)  &
    p\otimes (1-p)\otimes (1-q)\\
    +(1-p)\otimes q\otimes(1-p) &
    +(1-p)\otimes (1-q)\otimes (1-q) &
    +(1-p)\otimes q\otimes p &
    +(1-p)\otimes (1-q)\otimes q\\
    &&&\\
    &&&\\
    (1-p)\otimes p\otimes p &
    (1-p)\otimes (1-p)\otimes q &
    (1-p)\otimes p\otimes (1-p) &
    (1-p)\otimes (1-p)\otimes (1-q)\\
    +p\otimes q\otimes(1-p) &
    +p\otimes (1-q)\otimes (1-q) &
    +p\otimes q\otimes p &
    +p\otimes (1-q)\otimes q\\
    &&&\\
    &&&\\
    q\otimes (1-p)\otimes p &
    q\otimes p\otimes q &
    q\otimes (1-p)\otimes (1-p) &
    q\otimes p\otimes (1-q)\\
    +(1-q)\otimes (1-q)\otimes(1-p) &
    +(1-q)\otimes q\otimes (1-q) &
    +(1-q)\otimes (1-q)\otimes p &
    +(1-q)\otimes q\otimes q\\
    &&&\\
    &&&\\
    (1-q)\otimes (1-p)\otimes p &
    (1-q)\otimes p\otimes q &
    (1-q)\otimes (1-p)\otimes (1-p) &
    (1-q)\otimes p\otimes (1-q)\\
    +q\otimes (1-q)\otimes(1-p) &
    +q\otimes q\otimes (1-q) &
    +q\otimes (1-q)\otimes p &
    +q\otimes q\otimes q\\
\end{pmatrix}}

\begin{document}
\title{A Concrete Model for the Quantum Permutation Group on 4 Points}
\author{Nicolas Faroß and Moritz Weber }
\address{Nicolas Faroß and Moritz Weber, Saarland University, Fachbereich Mathematik, Postfach 151150,
66041 Saarbr\"ucken, Germany}
\email{faross@math.uni-sb.de, weber@math.uni-sb.de}
\date{\today}

\begin{abstract}
In 2019, Jung-Weber gave an example of a concrete magic unitary $M$,
which defines a $C^*$-algebraic model of the quantum permutation
group $S_4^+$. We show with the help of a computer that there exist no
polynomials up to degree $50$ separating the entries of $M$ from the
generators of $C(S_4^+)$. This indicates that the magic unitary $M$ might already
define a faithful model of $S_4^+$.
\end{abstract}

\maketitle

\section{Introduction}\label{section:introduction}

\noindent The quantum permutation group $S_n^+$ was first introduced by Wang
in~\cite{wang} and it can be regarded as a generalization
of the classical symmetric group $S_n$ in the sense of Woronowicz's compact matrix quantum
groups (see~\cite{woronowicz}). 
It is defined via the universal $C^*$-algebra 
\[
    C(S_n^+) := C^*(u_{ij},\, 1 \leq i,j \leq n ~|~ \text{$u$ is a magic unitary}),
\]
where a matrix $u = (u_{ij})$ is a magic unitary if its entries satisfy the 
relations
\begin{align*}
    u_{ij}^2 = u_{ij}^* = u_{ij},
    &&
    \sum_{k=1}^n u_{ik} = 1,
    &&
    \sum_{k=1}^n u_{kj} = 1
    &&
    (1 \leq i, j \leq n).
\end{align*}
Note that magic unitaries with entries in $\C$ are exactly 
permutation matrices, which justifies the name quantum permutation group.
In~\cite{jung}, Jung and the second author gave an example of a concrete magic unitary,
which defines a model of the quantum permutation group $S_4^+$. It is given by
\[
    M = \scalebox{0.6}{$\MagicUnitary$},
\]
where $p$ and $q$ are universal projections satisfying $p^2 = p^* = p$
and $q^2 = q^* = q$. However, it remained open 
if the model is faithful, i.e.\ if the $*$-homomorphism
\[
    \varphi \colon C(S_4^+) \to B, ~ u_{ij} \mapsto M_{ij}
\]
is injective, where $B$ denotes the $C^*$-algebra generated by the entries $M_{ij}$. 
In particular, Jung-Weber were interested in the 
existence of a non-commutative polynomial $P$ such that
$P$ vanishes in the entries of $M$ but $P$ does not vanish in the generators of $C(S_4^+)$.
Such a polynomial would prove that the mapping $\varphi$ is not injective.
Our main result now partially answers this question to the negative by showing
with the help of a computer (using GAP:GBNP \cite{GBNP} and SageMath \cite{sagemath}) that there exists no such polynomial up to degree $50$.

\begin{theorem*}[\Cref{theorem:main}]
Let $P \in \npoly{X}$ be a non-commutative polynomial in the entries of a $4 \times 4$ matrix 
$X = (x_{ij})$ with $\deg P \leq 50$. If
$P(M) = 0$ then $P$ is contained in the ideal generated by the
magic unitary relations
\begin{align*}
    x_{ij}^2 &= x_{ij}, && (1 \leq i, j \leq 4) \\
    \sum_{k=1}^n x_{ik} &= \sum_{k=1}^n x_{kj}  = 1,  && (1 \leq i, j \leq 4) \\
    x_{ij} x_{ik} &= x_{ji} x_{ki} = 0.  && (1 \leq i, j, k \leq 4, j \neq k)
\end{align*}
\end{theorem*}

\noindent Since the generators of $C(S_4^+)$ also satisfy the relations of the previous theorem, every
polynomial up to degree $50$ vanishing in the entries of $M$ has to vanish in the generators of $C(S_4^+)$.
Hence, we immediately obtain the following corollary.

\begin{corollary*}[\Cref{corollary:main}]
Denote with $u$ the generators of $C(S_4^+)$ and let $P \in \npoly{X}$ be a polynomial with $\deg P \leq 50$.
Then $P(M) = 0$ if and only if $P(u) = 0$.
\end{corollary*}

\noindent Note that the degree $50$ in the previous results is an arbitrary bound, 
which can be increased by providing more time and space to our algorithm.
Further, the maximal degree $50$ is quite large, which indicates that there 
exists no polynomial at all vanishing in the entries of $M$ but not vanishing in 
the generators of $u$. In this case, the concrete magic unitary $M$ might define 
a faithful $C^*$-algebra model of the quantum permutation group $S_4^+$.

We will start in \Cref{section:models-of-qpg} with the definition and some facts about the quantum permutation group, before 
we recall the construction of the concrete magic unitary $M$ from above. 
Then we formulate our main theorem in \Cref{section:computing-separating-polynomials}
and give an overview of the algorithm for proving it. More details and the computation results 
are then provided in \Cref{section:quotient-basis} and \Cref{section:matrix-elimination}.
In \Cref{section:concluding-remarks}, we present
further arguments for why $M$ could be a faithful model of $S_4^+$
and discuss our result in the context of quantum groups and other models 
of $S_n^+$ \cite{banica07}, \cite{banica09b}, \cite{banica15}, \cite{banica17}.
Additionally, we consider the case of larger $n$ and show that our result no longer
holds for similar models of $S_n^+$ with $n > 4$. 

\section*{Acknowledgements}
\noindent The second author has been supported by the SFB-TRR 195 (this work is a contribution to the SFB-TRR 195), the Heisenberg program of the DFG and OPUS-LAP \emph{Quantum groups, graphs and symmetries via representation theory}. This work has been part of the 
first author's Bachelor's thesis.

\section{Models of the Quantum Permutation Group}\label{section:models-of-qpg}

\noindent We begin with the definition of magic unitaries and the quantum permutation
group $S_n^+$, before we come to models of $S_n^+$ and the construction of the concrete magic unitary $M$.
These definitions are formulated in the language of (universal) $C^*$-algebras, which are 
complex associative algebras $A$ with an involution $* \colon A \to A$ and a compatible norm satisfying
the $C^*$-identity
$
    \lVert a \rVert^2 = \lVert a^*a \rVert.
$
However, we will not use this extra structure for the most part and consider a $C^*$-algebra just as 
an (not necessarily commutative) associative complex algebra.
For the general theory of $C^*$-algebras, we refer to~\cite{blackadar06} and for 
an introduction to universal $C^*$-algebra see~\cite{li}.

\begin{definition}[Magic unitary]\label{definition:magic-unitary}
Let $A$ be a unital $C^*$-algebra and $M \in M_n(A)$.
The matrix $M$ is called a \textit{magic unitary} if
its entries are projections and each row and each column sums up to $1$, i.e.
\begin{align*}
    M_{ij}^2 = M_{ij}^* = M_{ij},
    &&
    \sum_{k=1}^n M_{ik} = 1,
    &&
    \sum_{k=1}^n M_{kj} = 1
    &&
    (1 \leq i, j \leq n).
\end{align*}
\end{definition}

\begin{definition}[Quantum permutation group]\label{definition:quantum-permutation-group}
Let $u = (u_{ij})$ be a $n \times n$-matrix of generators and define
the universal unital $C^*$-algebra
\begin{align*}
    A := C^*(u_{ij},\, 1 \leq i,j \leq n ~|~ \text{$u$ is a magic unitary}).
\end{align*}
Then $S_n^+ := (A, u)$ is called the \textit{quantum permutation group}
of size $n$ \cite{wang}. Further, we denote the $C^*$-algebra $A$ with $C(S_n^+)$.
\end{definition}

\noindent We refer to~\cite{weber} for more information on magic unitaries and an overview of some related open problems. Note that there exists a $*$-homomorphism 
\[
    \Delta \colon C(S_n^+) \to C(S_n^+) \otimes C(S_n^+), \ u_{ij} \mapsto \sum_{k=1}^n u_{ik} \otimes u_{kj},
\] 
called comultiplication, which turns $S_n^+$ into a compact matrix quantum group
in the sense of Woronowicz~\cite{woronowicz}. See for example \cite{levandovskyy} for a short introduction 
to quantum symmetries in the context of computer algebra and \cite{timmermann08}, \cite{neshveyev13} for
compact quantum groups in general.
However, we will be mainly interested in the algebraic
properties of $C(S_n^+)$ and we will only come back to the quantum group structure at the end in 
\Cref{section:concluding-remarks}. In addition to the defining relations, magic 
unitaries satisfy further relations, which are implied by the $C^*$-algebraic structure.

\begin{proposition}\label{proposition:magic-unitary-extra-relations}
Let $A$ be a unital $C^*$-algebra and $M  \in M_n(A)$ a magic unitary.
Then the product of two different entries in the same row or column is
zero, i.e.
\begin{align*}
    M_{ij} \cdot M_{ik} = 0, &&
    M_{ji} \cdot M_{ki} = 0  &&
    (1 \leq i, j, k \leq n,\, j \neq k).
\end{align*}
\end{proposition}
\begin{proof}
By multiplying the relation
$
    \sum_{k=1}^n M_{ik} = 1
$
with $M_{ij}$ from both sides, we infer that 
$
    \sum_{k \neq j} M_{ij} M_{ik} M_{ij} = 0
$
is a sum of positive elements. By the theory of $C^*$-algebras, each of these summands
must be zero and hence
\[
    \lVert M_{ik} M_{ij} \rVert^2 
    = \lVert {(M_{ik} M_{ij})}^* (M_{ik} M_{ij}) \rVert
    = 0
\]
by the $C^*$-identity.
\end{proof}

\noindent Now, consider an arbitrary magic unitary $M \in M_n(A)$ and the $C^*$-subalgebra
$B \subseteq A$ generated by the entries of $M$. Then
the universal property of $C(S_n^+)$ implies the existence of a surjective
$*$-homomorphism $\varphi \colon C(S_n^+) \to B$ mapping the generators
$u_{ij}$ to $M_{ij}$. Hence, every pair $(B, M)$ of a magic unitary $M$ and 
its corresponding $C^*$-algebra $B$ defines a $C^*$-algebraic 
model of the quantum permutation group $S_n^+$. However, $(B, M)$ is not 
necessarily a compact matrix quantum group
since there might not exist a comultiplication $\Delta \colon B \to B \otimes B$.

In~\cite{jung}, Jung-Weber constructed sequences $(B_i, M_i)$
of such $C^*$-algebraic models. These are obtained by starting with an initial magic unitary 
and iterating the $\operp$-operator from~\cite{woronowicz}. For two 
matrices $M_1 \in M_n(A_1)$ and $M_2 \in M_n(A_2)$ this operator yields a new matrix
$M_1 \operp M_2 \in M_n(A_1 \otimes A_2)$ with entries given by 
\begin{align*}
    {(M_1 \operp M_2)}_{ij} = \sum_{k=1}^n {(M_1)}_{ik} \otimes {(M_2)}_{kj} && (1 \leq i, j \leq n).
\end{align*}
One can directly check that the matrix $M_1 \operp M_2$ is a magic unitary 
if both $M_1$ and $M_2$ are magic unitaries. Hence, starting with an initial $n \times n$ magic unitary
$R$, one obtains a sequence
\[
   (B_1, R), \ (B_2, R^{\smalloperp 2}), \ (B_3, R^{\smalloperp 3}), \ \dots
\]
of models of $S_n^+$, where each $B_i$ is the $C^*$-algebra generated by the entries of $R^{\smalloperp i}$.
Further, Jung-Weber constructed suitable initial matrices $R$ from which one can reconstruct $S_n^+$ as an inverse limit.
One such initial magic unitary is given by
\[
    R :=
    \left(\begin{matrix}
        p & 0 & 1-p & 0 \\
        1-p & 0 & p & 0 \\
        0 & q & 0 & 1-q \\
        0 & 1-q & 0 & q
    \end{matrix}\right)
    \in M_4(A),
\]
where 
\[
    A := C^*\left(p, q \ \middle| \ p^2 = p^* = p,\, q^2 = q^* = q \right) = C^*( \Z_2 * \Z_2)
\]
is the universal unital $C^*$-algebra generated by two projections.
When iterating the $\operp$-product of $R$ with itself, one obtains 
\[
    R^{\smalloperp 2}
    = \scalebox{0.8}{$
\begin{pmatrix}
\parbox{14.3cm}{
\begin{tabular}{r @{\hspace{2pt}}c@{\hspace{2pt}} l @{\hspace{20pt}} r @{\hspace{2pt}}c@{\hspace{2pt}} l @{\hspace{20pt}} r @{\hspace{2pt}}c@{\hspace{2pt}} l @{\hspace{20pt}} r @{\hspace{2pt}}c@{\hspace{2pt}} l }
$p$&$\otimes$& $p$&$(1-p)$&$\otimes$& $q$&$p$&$\otimes$& $(1-p)$&$(1-p)$&$\otimes$& $(1-q)$\\[4pt]
$(1-p)$&$\otimes$& $p$&$p$&$\otimes$&$q$&$(1-p)$&$\otimes$ &$(1-p)$&$p$&$\otimes$&$(1-q)$\\[4pt]
$q$&$\otimes$& $(1-p)$&$(1-q)$&$\otimes$& $(1-q)$&$q$&$\otimes$& $p$&$(1-q)$&$\otimes$&$q$\\[4pt]
$(1-q)$&$\otimes$& $(1-p)$&$q$&$\otimes$& $(1-q)$&$(1-q)$&$\otimes$& $p$&$q$&$\otimes$& $q$
\end{tabular}
}
\end{pmatrix}
$}
\]
and for $R^{\smalloperp 3}$ the magic unitary from the introduction.

\begin{definition}\label{definition:matrix-M}
In this article, let
\[
    M := \scalebox{0.6}{$\MagicUnitary$}.
\]
denote our \textit{concrete magic unitary}. It is given as $M = R^{\smalloperp 3}$ with 
entries in $A \otimes A \otimes A$, where $A$ is the universal $C^*$-algebra genereted by two projections as defined above.
\end{definition}

\noindent Note that we use the same letter for all three generators $p$ in the first, second and third tensor leg of $A\otimes A\otimes A$, for notational simplicity, rather than writing $p_1\otimes p_2\otimes p_3$; and likewise for the letter $q$.
In the previous setting, Jung-Weber asked for the existence of polynomials $P_n$, which separate 
these models.

\begin{question}[\cite{jung}]\label{question:open-question-specific}
Are there polynomials ${(P_n)}_{n \in \N}$ such that $P_n(R^{\smalloperp n}) = 0$
and $P_n(R^{\smalloperp n+1}) \neq 0$?
\end{question}

\noindent Such polynomials exist for $n < 3$ and are for example given by 
$P_1 = x_{12}$ and $P_2 = x_{12}x_{24}$ as can be verified directly (see also~\cite{jung}).
However, the case $n \geq 3$ remained open.
Our main result now shows that any polynomial $P$ with $\deg P \leq 50$, which vanishes
in $R^{\smalloperp 3}$, lies in the ideal generated by the relations of a
magic unitary. Hence, it also vanishes in the entries of $R^{\smalloperp 4}$ and the 
generators of $C(S_4^+)$. This answers the above question in the negative for polynomials up to 
degree $50$: There exists no polynomial $P_3$ with $\deg P_3 \leq 50$ such that $P_3(R^{\operp 3}) = 0$
and $P_3(R^{\operp 4}) \neq 0$. Further, it indicates that there exists no
polynomial at all separating the entries of $R^{\smalloperp 3}$
and the entries of $R^{\smalloperp 4}$.
In this case, $R^{\smalloperp 3}$ might already define a faithful model of $S_4^+$. 
For more details on these models, we refer again to~\cite{jung}.

\section{Computing Separating Polynomials}\label{section:computing-separating-polynomials}

\noindent In the following, we introduce some notations in order to define separating polynomials 
and formulate our main theorem. Let $X$ be a finite set. Then denote with $\npoly{X}$ 
the free associative unital algebra on $X$. Its elements can be regarded as non-commutative
polynomials in the variables $X$. Further, we will require every ideal $I \subseteq \npoly{X}$ 
to be two-sided. Since we are interested in magic unitaries, define the variables
\[
    X_n := \{x_{11}, x_{12}, \hdots, x_{nn} \},
\]
as entries of a general $n \times n$ matrix.
Note that it is sufficient to consider only the variables $x_{ij}$ and omit the $x_{ij}^*$
because the entries of a magic unitary are self-adjoint.
Given a matrix $M \in M_n(A)$ over some algebra $A$, we denote with $\varphi_M$ the 
substitution homomorphism
\[
    \varphi_M \colon \npoly{X_n} \to A, ~x_{ij} \mapsto M_{ij}.
\]
Further, we define the ideal generated by the magic unitary relations
from \Cref{definition:magic-unitary} and \Cref{proposition:magic-unitary-extra-relations}.

\begin{definition}[Magic unitary ideal]
Let $n \in \N$. Then define the \textit{magic unitary ideal} $I_n \subseteq \npoly{X_n}$,
which is generated by the polynomials
\begin{align*}
    x_{ij}^2 - x_{ij}, && \sum_{k=1}^n x_{ik} - 1, && \sum_{k=1}^n x_{kj} - 1
    && (1 \leq i, j \leq n), \\
    x_{ij} \cdot x_{ik}, &&
    x_{ji} \cdot x_{ki}  && &&
    (1 \leq i, j, k \leq n,\, j \neq k).
\end{align*}
\end{definition}

\noindent Using the previous definitions, we can now introduce
separating polynomials and formulate our main result.

\begin{definition}[Separating polynomial]
Let $A$ be a unital $C^*$-algebra and $M \in M_n(A)$ be a magic unitary. A non-zero polynomial $P \in \npoly{X_n}$
is called \textit{separating} if $\varphi_M(P) = 0$ but $P \notin I_n$.
\end{definition}

\begin{theorem}\label{theorem:main}
There exists no separating polynomial $P \in \npoly{X_4}$ with $\deg P \leq 50$ for the concrete magic unitary $M$
from \Cref{definition:matrix-M}.
\end{theorem}

\noindent In the following, we will outline our approach for proving \Cref{theorem:main}
with the help of a computer\footnote[1]{An implementation of our algorithms can be found at \href{https://github.com/nfaross/model-s4plus}{https://github.com/nfaross/model-s4plus}}. More details and the computation results are then presented
in \Cref{section:quotient-basis} and \Cref{section:matrix-elimination} before we come to the final proof 
in \Cref{section:main-proof}. 
A discussion of our result in the setting of models of $S_4^+$ can then be found in \Cref{section:concluding-remarks},
where we also show that a generalization of \Cref{theorem:main} does not hold
in the case $n > 4$.

To compute separating polynomials for the concrete magic unitary $M$ from \Cref{definition:matrix-M},
we consider the substitution homomorphism $\varphi_M$.
Since $M$ is a magic unitary, we have $I_4 \subseteq \ker \varphi_M$, such that
$\varphi_M$ can be factored as
\begin{center}
\begin{tikzcd}
    \npoly{X_4}  \arrow[rd, "\pi"] \arrow[rr, "\varphi_M"] & &
    A^{\otimes 3}  \\ &
    \npoly{X_4} / I_4 \arrow[ru, "\psi"] &
\end{tikzcd}
\end{center}
Here, $A$ denotes the universal $C^*$-algebra generated by two projections $p$ and $q$.
Further, denote with
\[
    V_m := \{ \pi(P) ~|~ P \in \npoly{X_4}, \deg P \leq m \}
\]
the subspace of all residue classes of polynomials up to degree $m$.
Then the kernel of the restriction $\psi|_{V_m}$ contains exactly the residue
classes of separating polynomials up to degree $m$.
Hence, there exists no separating polynomial up to degree $m$ if and only if
$\ker \psi|_{V_m} = \{ 0 \}$.
To show this statement, we proceed in two steps.
\begin{enumerate}
\item In \Cref{section:quotient-basis}, we construct a basis $\B_m$ for the subspace $V_m$, 
which can be obtained from a Gröbner basis for the magic unitary ideal $I_4$.
It turns out that such a basis $\B_m$ can alternatively be described by a finite automaton, 
which is a special kind of labelled graph. This finite automaton then allows us to efficiently 
enumerate all basis elements and compute the dimension of $V_m$. In particular, we obtain 
that $\abs{\B_m} = \Theta(m^3)$, such that the dimension grows only polynomial in the degree $m$.
Note that this polynomial growth is essential and is required for performing the following computations with a large degree.
\item In \Cref{section:matrix-elimination}, we use the basis $\B_m$ to construct a transformation
matrix $\Psi_m$ of the mapping $\psi|_{V_m}$. Using a special form of Gaussian Elimination
we are then able to compute a lower bound on the rank of $\Psi_m$. 
By running these algorithms for $m = 50$, we obtain
in \Cref{section:main-proof} that $\ker \psi|_{V_{50}} = \{0\}$, which proves \Cref{theorem:main}.
Further, we analyze 
these algorithms and show that the matrix construction and elimination have a time and space 
complexity of $\OO(m^6)$. Hence, it would be possible to increase $m$ by providing more time and space.
\end{enumerate}

\section{Constructing a Quotient Basis}\label{section:quotient-basis}

\noindent We start with some facts about Gröbner bases and finite automaton
in order to show how these can be used to describe a basis $\B$ of a quotient $\npoly{X} / I$. 
These results will then be applied to the spaces $V_m \subseteq \npoly{X_4}/I_4$ from 
the previous section. Note that the results in this section are not new and 
labelled graphs were for example used by Ufnarovski\u{\i} in~\cite{ufnarovski} to 
describe bases of such quotients. However, we will restate them for convenience.

\subsection{Gröbner bases}\label{section:groebner-bases}

Gröbner bases generalize Euclidean division to multivariate polynomial rings and allow for example
to solve the ideal membership problem. See~\cite{mora} for a detailed introduction in the
commutative and non-commutative case.
In the following, denote with $X$ a finite set of generators. Further, let
$\mon{X}$ be the set of all monomials in $\npoly{X}$ including $1$. Before we can 
define Gröbner bases, we first need a well-ordering on $\mon{X}$.

\begin{definition}[Degree lexicographic order]
Let $\leq$ be a well-ordering on $X$. Then one can extend it to a well-ordering on $\mon{X}$
by first comparing the degree of two monomials. If the degree is equal, then
monomials are compared lexicographically from left to right. This ordering
is called \textit{degree lexicographic order}. 
\end{definition}

\noindent In the following, we fix some well-ordering on $X$ and equip 
$\mon{X}$ with the degree lexicographic order from the previous definition. This allows us to define 
the leading term of a polynomial.

\begin{definition}[Leading term]
Let $P \in \npoly{X}$ be a non-zero polynomial, which can uniquely be written as
$P = \sum_{i=1}^n \alpha_i x_i$
for some $a_1, \ldots, \alpha_n \in \C \setminus \{ 0 \}$ and $x_1, \ldots, x_n \in \mon{X}$.
Then the \textit{leading term} $\LT(P)$ is the largest monomial $x_i$ with respect to the degree 
lexicographic order in this representation. In this case, the degree $\deg P$ is given by the length of
$\LT(P)$.
\end{definition}

\noindent With these definitions we can now introduce Gröbner bases.

\begin{definition}[Gröbner basis]
Let $I \subseteq \npoly{X}$ an ideal. A finite set $G \subseteq I$
of non-zero polynomials is called \textit{Gröbner basis} for the ideal $I$ if
the leading terms of $I$ are exactly the monomials divisible by a leading term from $G$, i.e.
\[
    \{ \LT(P) ~|~ P \in I,~P \neq 0 \} = \{ a\LT(g)b ~|~ a, b \in \mon{X}, g \in G \}.
\]
\end{definition}

\noindent Given a Gröbner basis $G \subseteq I$, the next lemma shows that there exists
a simple description of a basis for $\npoly{X}/I$.

\begin{lemma}\label{lemma:quotient-basis}
Let $I \subseteq \npoly{X}$ be an ideal and $G \subseteq I$ a Gröbner basis.
Then a basis of the quotient $\npoly{X}/I$ is given by the residue classes of
\[
    \B = \mon{X} \setminus \{ a\LT(g)b ~|~ a, b \in \mon{X}, g \in G \}.
\]
Moreover, the residue classes of
\[
    \B_m = \{ P \in \B ~|~ \deg P \leq m \}
\]
form a basis for the spaces $V_m \subseteq \npoly{X}/I$ of residue classes of 
polynomials up to degree $m$.
\end{lemma}
\begin{proof}
Note that the elements of $\B$ are linearly independent since they are given by distinct 
monomials. Further, we can write
\[
    \npoly{X} = I \oplus \Span{B},    
\]
where $\Span{B}$ denotes the linear span of $\B$. See~\cite[Theorem 1.3]{mora} 
in combination with the definition of Gröbner basis for a proof of this statement.
Taking the quotient by $I$, we immediately obtain that $\npoly{X}/I \cong \Span{B}$
and that the residue classes of $B$ form a basis of $\npoly{X}/I$. For the second part 
of the lemma, one checks that if $P = P_1 + P_2$ with $P_1 \in I$ and $P_2 \in \Span{B}$,
then $\deg P_2 \leq \deg P$, which implies that the space $V_m$ is spanned 
by the residue classes of $\B_m$.
\end{proof}

\subsection{Finite Automata}
Next, we want to reformulate \Cref{lemma:quotient-basis} in the language of finite automata.
These are a fundamental tool in theoretical computer science
and can be used to describe sets of words over an alphabet. 
In the following, we start by recalling the definition of language and finite automata. A detailed introduction to languages and
automata can for example be found in~\cite{hopcroft}.

\begin{definition}
Let $\Sigma$ be a finite set of symbols called \textit{alphabet}.
A \textit{word} over $\Sigma$ is a finite sequences $w_1 \cdots w_n$ of
symbols $w_1, \hdots, w_n \in \Sigma$.
Denote with $\cleene{\Sigma}$ the Kleene closure of $\Sigma$, which is the set of all words over $\Sigma$
including the empty word $\varepsilon$. Then a set of words $L \subseteq \cleene{\Sigma}$
is called \textit{language}.
\end{definition}

\begin{remark}\label{remark:monomial-words}
Consider the algebra of non-commutative polynomials $\npoly{X}$. Then the set of variables
$X$ can be considered as an alphabet. In the previous notation,
the set $\mon{X}$ of all monomials
coincides with the Kleene closure $\cleene{X}$,
if we identify the empty word $\varepsilon$ with the unit $1$.
\end{remark}

\noindent Next, we introduce finite automata, which can be used to describe
a special class of languages called regular languages.

\begin{definition}[Finite automaton]
A \textit{finite automaton} over an alphabet $\Sigma$ is a directed and labelled
graph $\Gamma = (V, E, \ell, s_0, F)$ where
\begin{enumerate}
\item $V$ denotes the set of vertices and $E$ the set of directed edges.
\item $\ell \colon E \to \Sigma$ assigns to each edge in $E$ a symbol from $\Sigma$.
\item $s_0 \in V$ is the \textit{initial state}.
\item $F \subseteq V$ is a set of \textit{final states}.
\end{enumerate}
Note that multi-edges are allowed and the set $F$ can be empty.
The vertices are also called \textit{states} and the edges \textit{transitions}.
\end{definition}

\begin{definition}
Let $\Gamma$ be a finite automaton over an alphabet $\Sigma$.
We say $\Gamma$ \textit{accepts} a word $w \in \cleene{\Sigma}$, if 
there exists a directed path $e_1, \hdots, e_n \in E$ starting at the initial state $s_0$ 
and ending at a final state $s \in F$, such that 
\[
    w = \ell(e_1) \cdots \ell(e_n).
\]
\end{definition}

\begin{definition}[Regular language]
A language $L \subseteq \cleene{\Sigma}$ is called \textit{regular},
if there exists a finite automaton $\Gamma$ over $\Sigma$ such that 
\[
    L = \{ w \in \cleene{\Sigma} ~|~ \text{$\Gamma$ accepts $w$} \}.
\]
In this case, we write $L = L(\Gamma)$ and call it the 
language accepted by $\Gamma$.
\end{definition}

\begin{example}\label{example:finite-automaton}
\Cref{figure:simple-finite-automaton} shows an example of a finite automaton $\Gamma$ over the alphabet 
$\Sigma = \{x, y, z \}$. Its initial state $s_0 = 1$ is marked with an arrow
and the final states $F = \{ 3 \}$ are circled. One can check that its accepted language is given by 
\[
    L(\Gamma) = \{ w \in \cleene{\Sigma} ~|~ \text{$w$ contains $xz$ as subword} \}.
\]
Note that multiple edges with labels $x$, $y$ and $z$ at the states $1$ and $3$ are drawn as one edge.
\begin{figure}
\begin{center}
\begin{tikzpicture}
\node[state, initial] (v1) at (0, 0) {$1$};
\node[state] (v2) at (3, 0) {$2$};
\node[state, accepting] (v3) at (6, 0) {$3$};
\path[->] (v1) edge node[above] {$x$}(v2);
\path[->] (v2) edge node[above] {$z$}(v3);
\path[->, loop/.style={looseness=10}] (v1) edge[loop above] node {$x, y, z$}(v1);
\path[->, loop/.style={looseness=10}] (v3) edge[loop above] node {$x, y, z$}(v3);
\end{tikzpicture}
\end{center}
\caption{A simple finite automaton with $3$ states.}\label{figure:simple-finite-automaton}
\end{figure}
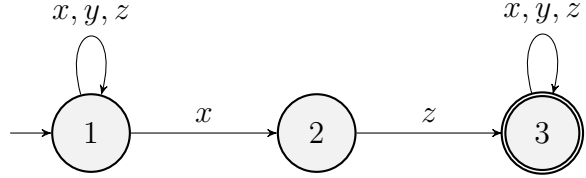
\end{example}

\noindent Using the previous definition of finite automata and regular
language, we can now formulate and prove the following lemma.

\begin{lemma}\label{lemma:subword-automaton}
Let $\Sigma$ be an alphabet and $S \subseteq \cleene{\Sigma}$ a finite set of
words. Then the set
\[
    \{ w \in \cleene{\Sigma} ~|~ \text{$w$ contains no word from $S$ as subword} \}
\]
is a regular language.
\end{lemma}
\begin{proof}
Let $s \in S$ and define the sets
\[
    L_s = \{ w \in \cleene{\Sigma} ~|~ \text{$w$ contains $s$ as subword} \}.
\]
These are regular languages since the corresponding finite automaton can be 
constructed similar to the one in \Cref{example:finite-automaton}.
Since regular languages are closed under unions and complements 
(see~\cite[Chapter 2]{hopcroft}), the following language is again regular:
\begin{align*}
    \cleene{\Sigma} \setminus \bigcup_{s \in S} L_s
    = \{ w \in \cleene{\Sigma} ~|~ \text{$w$ contains no word from $S$ as subword} \}.
\end{align*}
\end{proof}

\noindent By combining the previous lemma with \Cref*{remark:monomial-words}, we now obtain the 
following version of \Cref{lemma:quotient-basis}, which allows us to describe
quotient bases by finite automata.

\begin{lemma}\label{lemma:quotient-basis-automaton}
Let $I \subseteq \npoly{X}$ be an ideal and $G$ a Gröbner basis
for $I$. Then there exists a finite automaton $\Gamma$ over the
alphabet $X$ such that the residue classes of the accepted language $L(\Gamma)$
are a basis for $\npoly{X}/I$. \\
In particular, the residue classes of all accepted words up to length $m$
are a basis for $V_m \subseteq \npoly{X}/I$, where
$V_m$ denotes the subspace of residue classes of polynomials up to degree $m$.
\end{lemma}
\begin{proof}
As in \Cref{remark:monomial-words}, we identify the set of monomials $\mon{X}$
with words over the alphabet $X$. Define $S = \{ \LT(g) ~|~ g \in G \}$. Then $\B$
from \Cref{lemma:quotient-basis} can be written as
\[
    \B = \{ w \in \cleene{X} ~|~ \text{$w$ contains no word from $S$ as subword} \}.
\]
Using \Cref{lemma:subword-automaton} and the definition of a regular language, 
we obtain a finite automaton $\Gamma$
such that the residue classes of $L(\Gamma) = \B$ are a basis of $\npoly{X}/I$.
Since the degree $\deg w$ corresponds to the length of a word $w \in \mon{X}$,
we further obtain that the sets $\B_m$ from \Cref{lemma:quotient-basis}
are given by all accepted words up to length $m$.
\end{proof}

\subsection{Computational Results}\label{section:computatial-results}
In the following, we present our results for applying \Cref{lemma:quotient-basis-automaton}
to the magic unitary ideal $I_4$.
For computing a Gröbner basis $G$ for the magic unitary ideal $I_4$,
the computer algebra system GAP \cite{GAP4} and the package GBNP~\cite{GBNP} are used.
The corresponding finite automaton is then constructed using SageMath~\cite{sagemath}.
This is done as described in the proof of \Cref{lemma:subword-automaton} 
by computing the union and complement of simpler finite automatons\footnote[1]{An implementation of our algorithms can be found at \href{https://github.com/nfaross/model-s4plus}{https://github.com/nfaross/model-s4plus}}.
Further, it is possible to simplify the resulting finite automaton and minimize its number of states.
A picture of the final finite automaton for $I_4$ can be found \Cref{appendix:automaton}.
In particular, it has $17$ states and every state is final.
Further, SageMath allows us to compute that it contains exactly 
${(2m + 1)}^2$ accepting paths of length $m$, 
which implies that there are
\[
    \sum_{k=0}^m {(2k+1)}^2 = \frac{1}{6}(2m+1)(2m+2)(2m+3) = \binom{2m+3}{3}
\] 
accepting paths up to length $m$. Hence, we obtain that 
$\dim V_m = \Theta(m^3)$ grows only polynomial in $m$.
In addition to the ideal $I_4$, we were able to
compute Gröbner bases and construct finite automata for the magic unitary ideals 
$I_5$ and $I_6$. The resulting finite automata have $26$ and $37$ states respectively and
contain $\Theta(6.854\ldots^m)$ and $\Theta(13.928\ldots^m)$ accepting paths of length $m$.
Thus, the dimension of $V_m$ grows exponentially in these cases.

\section{Matrix Construction and Elimination}\label{section:matrix-elimination}

\noindent The goal of the following section is to construct a transformation
matrix $\Psi_m$ of the mapping $\psi|_{V_m}$ from \Cref{section:computing-separating-polynomials}
and to compute a lower bound on its rank in order 
to prove \Cref{theorem:main}.
Recall that the mapping $\psi$ was obtained by factoring 
the substitution homomorphism $\varphi_M$ of the concrete magic unitary $M$ from 
\Cref{definition:matrix-M} as follows.

\begin{center}
\begin{tikzcd}
    \npoly{X_4}  \arrow[rd, "\pi"] \arrow[rr, "\varphi_M"] & &
    A^{\otimes 3}  \\ &
    \npoly{X_4} / I_4 \arrow[ru, "\psi"] &
\end{tikzcd}
\end{center}

\noindent Here, $A$ is the universal unital $C^*$-algebra
generated by two projections, $I_4$ is the magic unitary ideal and 
$V_m \subseteq \npoly{X_4} / I_4$ is the subspace of residue
classes of polynomials up to degree $m$.

\subsection{Construction Algorithm}

In order to construct the transformation matrix for the mapping 
$\psi|_{V_m}$, we first have to 
choose a basis for its domain and its image. In \Cref{section:quotient-basis}
we constructed a finite automaton, which describes a basis $\B_m$ for 
the domain $V_m$. See \Cref{lemma:quotient-basis-automaton} for details on the construction
and \Cref{appendix:automaton} for the resulting automata.
For a basis of the image, consider the following sets
\begin{align*}
\A &:= \{ 1, p, q, pq, qp, pqp, \hdots \} \subseteq A, \\
\A_m &:= \{ a \in \A ~|~ \text{$a$ has at most $m$ factors} \} \subseteq A, \\
\A^{\otimes k}_m &:=
\{ a_1 \otimes \cdots \otimes a_k ~|~ a_1, \ldots, a_n \in \A_m \} \subseteq A^{\otimes k}.
\end{align*}

\noindent The next lemma shows that the image of $\psi|_{V_m}$ is contained in 
the linear span of $\A^{\otimes 3}_m$, which can then be chosen as a basis.

\begin{lemma}\label{lemma:image-dimension}
In the previous notation, the elements of $\A^{\otimes k}_m$ are linearly independent and $\psi(V_m) \subseteq \Span{\A^{\otimes 3}_m}$.
In particular, $\dim \psi(V_m) \leq {(2m+1)}^3$.
\end{lemma}
\begin{proof}
Since $p$ and $q$ are universal projections, the elements of $\A$ are linearly independent,
which also implies the linear independence of the sets $\A_m$ and $\A^{\otimes k}_m$. Further, 
if $a \in \A^{\otimes k}_m$ and $b \in \A^{\otimes k}_n$ then 
$ab \in \A^{\otimes k}_{m+n}$. Hence, if $P \in \npoly{X_4}$ is a polynomial of degree $m$,
then $P(M) \in \Span{A^{\otimes 3}_m}$, since the entries of the magic unitary $M$ are contained in 
$\Span{\A^{\otimes 3}_1}$. Thus, $\psi(V_m) \subseteq \Span{\A^{\otimes 3}_m}$. Further, 
we have $\abs{\A_m} = 2m+1$, which implies 
$\dim \psi(V_m) \leq \abs{A^{\otimes 3}_m} = {(2m+1)}^3$.
\end{proof}

\noindent Now, \Cref{algorithm:matrix-construction} can be used for constructing the transformation matrix $\Psi_m$ 
of the mapping $\psi|_{V_m}$ with respect to the basis $\B_m$ described
by the finite automaton $\Gamma$ from \Cref{appendix:automaton} and the basis $\A^{\otimes 3}_m$ 
of the previous paragraph. A proof of the correctness can be found 
in the next lemma.

\begin{algorithm}
\caption{Matrix construction}\label{algorithm:matrix-construction}
\begin{flushleft}
\textbf{Input:} degree $m$ \\
\textbf{Output:} matrix $\Psi_m$
\end{flushleft}
\smallskip
\begin{algorithmic}[1]
\State{initialize empty matrix $\Psi_m$}
\State{initialize queue $Q$ with $(s_0,\, 1 \otimes 1 \otimes 1,\, 0)$}
\While{$Q$ is not empty}
    \State{remove $(s,\, P,\, k)$ from $Q$}
    \State{insert column corresponding to $P$ into $\Psi_m$}
    \If{$k < m$}
        \ForAll{transition $s \to s'$ in $\Gamma$ with label $x_{ij}$}
            \State{$P' \gets P \cdot {(M)}_{ij}$}
            \State{insert $(s',\, P',\, k+1)$ into $Q$}
        \EndFor{}
    \EndIf{}
\EndWhile{}
\end{algorithmic}
\end{algorithm}

\begin{lemma}
In the previous notation, \Cref{algorithm:matrix-construction} constructs the transformation matrix $\Psi_m$
of the mapping $\psi|_{V_m}$ with respect to the bases $\B_m$ and $\A^{\otimes 3}_m$.
\end{lemma}
\begin{proof}
In order to construct the transformation matrix $\Psi_m$, we have to evaluate
\[
    \psi(\pi(x)) = \varphi_{M}(x)
    = \prod_{\ell=1}^k \varphi_{M}(x_{i_\ell j_\ell}) = \prod_{\ell=1}^k M_{i_\ell j_\ell}
\]
for each monomial $x = x_{i_1 j_1} x_{i_1 j_1} \cdots x_{i_k j_k}$ in the basis $\B_m$.
Each product then corresponds to a column in $\Psi_m$ when represented with
respect to $\A^{\otimes 3}_m$. This is done by traversing all paths up to length $m$ in the 
finite automaton $\Gamma$ from \Cref{appendix:automaton}, in order to generate all basis element 
$x \in \B_m$. Since each state in $\Gamma$ is final, each path corresponds exactly to an element 
$x \in \B_m$. Further, we directly multiply the corresponding $M_{ij}$
when generating a path, such that we obtain the columns of $\Psi_m$
in the same step.

More specifically, \Cref{algorithm:matrix-construction} maintains a queue of paths and traverses the 
finite automaton in a breadth-first search style.
In the queue, each path is represented by a triple $(s, P, k)$, where
$s$ is the last node of the path, $P \in \Span{\A^{\otimes 3}}$ is the polynomial
evaluated along the path and $k$ is the length of the path.
The algorithm starts with the triple $(s_0, 1 \otimes 1 \otimes 1, 0)$ and
in each step removes the next path from the queue and inserts the corresponding column
into $\Psi_m$.
Then all outgoing transitions from $s$ to $s'$ with label $x_{ij}$
are considered and new paths $(s', P', k+1)$ are added to the queue.
These new paths extend the current path along the transition $s \to s'$ and contain
the corresponding polynomial $P' = P \cdot {(M)}_{ij}$. In this way, every path up to 
length $m$ will be generated and the corresponding column will be inserted.
\end{proof}

\subsection{Elimination Algorithm}

Our next goal is to compute the rank of  $\Psi_m$ to prove \Cref{theorem:main}.
However, to do this efficiently, we have to
store the matrix $\Psi_m$ as a pair $\Psi_m = (rows, columns)$, where
$\textvar{rows}$ and $\textvar{columns}$ are maps such that
\begin{enumerate}
    \item $\textvar{rows}(i) = \{ j_1, \hdots, j_{n_i} \}$ is the set of non-zero
    columns $j_1, \hdots, j_{n_i}$ in row $i$,
    \item $\textvar{columns}(j) = \{ i_1, \hdots, i_{m_j} \}$ is the set of non-zero
    rows $i_1, \hdots, i_{m_j}$ in column $j$.
\end{enumerate}
Using this matrix format, \Cref{algorithm:matrix-elimination} now computes 
a lower bound on the rank of $\Psi_m$. Its correctness is proven in the following lemma.

\begin{algorithm}
\caption{Matrix elimination}\label{algorithm:matrix-elimination}
\begin{flushleft}
\textbf{Input:} sparse $n_r \times n_c$ matrix $\Psi_m = (\textvar{rows},\, \textvar{columns})$ \\
\textbf{Output:} lower bound $\textvar{rank}$ for the rank of $\Psi_m$
\end{flushleft}
\smallskip
\begin{algorithmic}[1]
\State{$\textvar{rank} \gets 0$}
\State{initialize empty stack $S$}
\For{$i = 1, \hdots, n_r$}
    \If{$\abs{\textvar{rows}(i)} = 1$}
        \State{ push $i$ to $S$ }
    \EndIf{}
\EndFor{}
\While{$S$ is not empty}
    \State{remove $i$ from $S$}
    \If{$\abs{\textvar{rows}(i)} = 1$}
        \State{$\{ j \} \gets \textvar{rows}(i)$}
        \For{$k \in \textvar{columns}(j)$}
            \If{$k \neq i$}
                \State{delete $j$ from $\textvar{rows}(k)$}
                \If{$\abs{\textvar{rows}(k)} = 1$}
                    \State{ push $k$ to $S$ }
                \EndIf{}
            \EndIf{}
        \EndFor{}
        \State{$\textvar{columns}(j) \gets \{ i \}$}
        \State{$\textvar{rank} \gets \textvar{rank} + 1$}
    \EndIf{}
\EndWhile{}
\end{algorithmic}
\end{algorithm}

\begin{lemma}
\Cref{algorithm:matrix-elimination} computes a lower bound
on the rank of the matrix $\Psi_m$.
\end{lemma}
\begin{proof}
\Cref{algorithm:matrix-elimination} performs a special form of Gaussian elimination and transforms $\Psi_m$
using elementary row operations. 
It searches rows $i$ which contain only one non-zero entry in some column $j$.
Next, all entries in column $j$ which occur in other rows $k \neq i$ are eliminated.
If such a row $k$ then contains only one non-zero remaining entry, it is pushed to a stack
to be considered next.
Each row with one non-zero entry is linearly independent of all other rows. 
Hence, the total number of such rows is a lower bound on the rank of the matrix $\Psi_m$.
Note that a row which was pushed to the stack could have been
eliminated before it gets processed. In this case, the row can be written
as a linear combination of other rows and does not contribute to the rank of the
matrix.
\end{proof}

\subsection{Proof of \Cref{theorem:main}}\label{section:main-proof}

\noindent Using \Cref{algorithm:matrix-elimination}, we can finally prove
\Cref{theorem:main}, which states that there exists no separating 
polynomial for the concrete magic unitary $M$ from \Cref{definition:matrix-M}
up to degree $50$.

\begin{proof}[Proof of \Cref{theorem:main}]
By running \Cref{algorithm:matrix-elimination}\footnote[1]{An implementation of our algorithms can be found at \href{https://github.com/nfaross/model-s4plus}{https://github.com/nfaross/model-s4plus}}, we obtain a lower
bound of $176851$ for the rank of $\Psi_{50}$, which also equals the 
number of columns (compare \Cref{section:computatial-results}).
Since the number of columns is the dimension of the image, the kernel of $\Psi_{50}$ is zero-dimensional.
Hence, there exist no separating polynomials for the concrete
magic unitary by the considerations in \Cref{section:computing-separating-polynomials}.
\end{proof}

\subsection{Complexity}
In the following, we analyse \Cref{algorithm:matrix-construction}
and \Cref{algorithm:matrix-elimination} and show that 
both have a complexity of $\OO(m^6)$.
Note that we fix the finite automaton $\Gamma$ and vary only the 
degree $m$. Further, we will assume that maps allow insertions and deletions 
in $\OO(1)$, which is approximately the case when implemented as hash maps. 

Before we come to the algorithms, we first have to consider the arithmetics in 
$\Span{A_m^{\otimes 3}}$. Note that an element $a \in \A_m$ is uniquely determined
by its length and parity. Hence, an element $a \in A_m^{\otimes 3}$ can be stored in constant
space and two elements $a, b \in A_m^{\otimes 3}$ can be multiplied in constant time.
For an element $a \in \Span{A_m^{\otimes 3}}$ denote with 
$\abs{a}$ the number of non-zero coefficients when represented with respect to 
the basis $A_m^{\otimes 3}$. Then $a$ can be stored using $\OO(\abs{a})$ space using a map which 
stores the corresponding coefficients for each basis element. Further, 
we can compute $ab$ for two elements $a, b \in \Span{A_m^{\otimes 3}}$
in time $\OO(\abs{a} \cdot \abs{b})$ by multiplying all basis elements pairwise.
With these considerations, we can now analyse our main algorithms.

\begin{lemma}\label{lemma:algorithm-complexity}
\Cref{algorithm:matrix-construction} and 
\Cref{algorithm:matrix-elimination}
have a complexity of $\OO(m^6)$.
\end{lemma}
\begin{proof}
First, we consider \Cref{algorithm:matrix-construction}. 
Since a queue allows all operations in constant time, its running time
is determined by the total time for multiplying polynomials and inserting rows into $\Psi_m$.
Since the entries of $M_3$ are constant, we can compute $P' = P \cdot {(M_3)}_{ij}$
in $\OO(\abs{P})$.
Further, a polynomial
$P$ can be inserted into $\Psi_m$ in $\OO(\abs{P})$ using the matrix format
described before. Hence, \Cref{algorithm:matrix-construction}
requires $\OO(N)$ time and space, where
\[
    N = \sum_{\text{$P$ computed}} \abs{P} = \sum_{x \in \B_m} \abs{\varphi_{M}(x)}
\]
is the number of non-zero entries in the matrix $\Psi_m$. 
Now, consider \Cref{algorithm:matrix-elimination}. Since
stacks allow all operations in $\OO(1)$ and each of the $\OO(N)$ non-zero entry of $\Psi_m$
is deleted at most once in constant time, \Cref{algorithm:matrix-elimination}
requires time $\OO(N)$. Hence, it remains to bound the number $N$ of non-zero entries 
of $\Psi_m$.
By \Cref{lemma:image-dimension}, we know that $\Psi_m$ has at most
$\abs{\A_m^{\otimes 3}} = (2m + 1)^3$ rows. On the other hand, the number
of columns is given $\abs{\B_m}$, which was computed in \Cref{section:computatial-results}.
Hence, $\Psi_m$ has at most 
\[
    N \leq {(2m + 1)}^3 \cdot \binom{2m+3}{3} = \OO(m^6)
\] 
non-zero entries.
\end{proof}

\section{Concluding remarks}\label{section:concluding-remarks}

\noindent In the following, we put our main result in the context of $C^*$-algebraic models 
of the quantum permutation group and discuss why the concrete magic unitary $M$ from \Cref{definition:matrix-M} might 
define a faithful model of $S_4^+$. Throughout this section, denote 
with $A$ again
the universal unital $C^*$-algebra generated by two projections and with 
$B$ the $C^*$-algebra generated by the entries of the concrete magic unitary $M$. 

\subsection{Hints for \texorpdfstring{$M$}{M} being a faithful model}

We start with an immediate consequence of 
\Cref{theorem:main}.

\begin{corollary}\label{corollary:main}
Let $u$ be the matrix containing the generators of $C(S_4^+)$ and let $P \in \npoly{X_4}$ be a polynomial with $\deg P \leq 50$.
Then $P(M) = 0$ if and only if $P(u) = 0$.
\end{corollary}
\begin{proof}
Let $P \in \npoly{X_4}$ be a polynomial with $\deg P \leq 50$.
Since $C(S_4^+)$ is the universal $C^*$-algebra generated by the entries of a $4 \times 4$
magic unitary, there exists a $*$-homomorphism $\Phi \colon C(S_4^+) \to B, u_{ij} \mapsto M_{ij}$.
Thus, if $P(u) = 0$, then 
\[
    P(M) = P(\Phi(u)) = \Phi(P(u)) = \Phi(0) = 0.
\]
On the other hand, if $P(M) = 0$, then $P \in I_4$ by \Cref{theorem:main}.
However, by the definition of $C(S_4^+)$ and \Cref{proposition:magic-unitary-extra-relations}, we have that 
$P(u) = 0$ for every $P \in I_4$.
\end{proof}

\noindent Since $50$ is a quite large bound for the degree of a non-commutative
polynomial, we conjecture that the previous corollary holds for polynomials
of an arbitrary degree. In particular, the bound $50$ arises from limited computational power
and there is no immediate reason why the setting of \Cref{theorem:main} should change when $m > 50$.
Also recall, that for the initial \Cref{question:open-question-specific}, we do have a polynomial $p_1$
of degree $1$ and a polynomial $p_2$ of degree $2$. Hence, it seems unreasonable that a polynomial $p_3$ 
would require a degree larger than $50$.
Further, the previous corollary also indicates the stronger statement that the concrete magic unitary $M$ 
from \Cref{definition:matrix-M} defines a faithful model
of $S_n^+$ in the sense that $C(S_n^+) \cong B$ via $u_{ij} \mapsto M_{ij}$.
In this case, it would be possible to obtain $C(S_n)$ as a quotient
of $B$, since $C(S_4)$ is the abelianization of $C(S_4^+)$.
The following proposition shows, that this necessary condition indeed holds.

\begin{proposition}\label{proposition:spectrum}
In the previous notation, $\Spec B \cong S_4$. In particular, $C(S_4)$ is a quotient of $B$.
\end{proposition}
\begin{proof}
Let $\varphi \in \Spec B$ be a character. Then $\varphi$ is uniquely determined by the magic unitary $\varphi(M) \in M_4(\C)$, which is a
permutation matrix. Hence, it remains to show that we obtain every permutation matrix via a character. 
Let $p_1, q_1, \ldots, p_3, q_3 \in \{0, 1\}$. Since $A$ is the universal $C^*$-algebra generated by two projections, there exists a $*$-homomorphism $\varphi \colon A^{\otimes 3} \to \C$ which
maps $p$ in the $i$-th tensor leg to $p_i$ and $q$ in the $i$-th tensor leg to $q_i$.
Restricting $\varphi$ to $B \subseteq A^{\otimes 3}$, we obtain a character $\varphi|_B \in \Spec B$.
\Cref{appendix:characters} then shows how to obtain every permutation matrix via a suitable 
choices of $p_1, q_1, \ldots, p_3, q_3$.
Thus, $\Spec B \cong S_4$ and by the Gelfand-Naimark Theorem we obtain a
surjective $*$-homomorphism $\Phi \colon B \to C(S_4)$. Hence, $C(S_4) \cong B/\ker \Phi$, 
which proves the second part of the statement.
\end{proof}

\noindent Note that \Cref{appendix:characters} shows that it is possible to construct 
each permutation matrix in the proof of \Cref{proposition:spectrum} by sending 
$p_0$ to zero. Hence, it is possible to obtain a slightly simpler model of $S_4^+$ with full spectrum 
by setting $p = 0$ in the first tensor leg in the entries of the concrete magic unitary $M$ from 
\Cref{definition:matrix-M}.

\subsection{Further results on \texorpdfstring{$S_n^+$}{Sn+}}

Firstly, recall that a $*$-algebra $A$ is said to be \textit{residually finite dimensional} if there
exists an injective $*$-homomorphism
\[
    \pi \colon A \to \prod_{i \in I} M_{n_i}(\C)
\] 
into a product of matrix algebras, where $I$ is an arbitrary (possibly infinite) index set.
In~\cite{brannan}, Brannan, Chirvasitu and Freslon showed
that the $*$-algebra $A$
corresponding to $S_n^+$ is residually finite dimensional. Hence, for each
$*$-polynomial $P \neq 0$ in the generators of $C(S_n^+)$ there
exists a $*$-homomorphism $\pi_i \colon A \to M_{n_i}(\C)$ such
that $\pi_{i}(P) \neq 0$, where $\pi_{i}$ is obtained by projecting onto the $i$-th
for some $i \in I$ depending on $p$. This gives to some extent some information
on possible models of $S_4^+$.

Secondly, one possible approach for showing that the concrete magic unitary $M$
from \Cref{definition:matrix-M} defines a faithful model of $S_4^+$ would be to show that 
$M$ is a corepresentation matrix of some compact quantum group $G$
with $B \subseteq C(G)$. In this case, the 
comultiplication of $C(G)$ restricts to a comultiplication on $B$, which 
turns $B$ into a compact matrix quantum group $H$ with $S_4 \subseteq H \subseteq S_4^+$.
Since the inclusion $S_4 \subseteq S_4^+$ is maximal (see~\cite{banica09}) 
and $B$ is non-commutative, it would follow that $H = S_4^+$.

Note that the $C^*$-algebra $A$ corresponds to the compact
quantum group $\widehat{\Z_2 \ast \Z_2}$, with comultiplication 
$\Delta \colon A \to A \otimes A$ given by 
\begin{align*}
    \Delta(p) &= p \otimes p + (1-p) \otimes (1-p), \\
    \Delta(q) &= q \otimes q + (1-q) \otimes (1-q).
\end{align*}
Hence, the tensor product $A^{\otimes 3}$ already carries a
direct product quantum group structure. However, 
it seems that one would have to find a different product structure in order to 
turn $M$ into a corepresentation matrix, which the authors were not able to achieve.

Thirdly, there are other models of $S_n^+$, which are studied for example in \cite{banica07}, \cite{banica09b}, \cite{banica15} and \cite{banica17}.
However, these are of different types and are constructed from Pauli matrices, Fourier matrices or some related constructions.
In the case of~\cite{banica07}, the corresponding model for $n = 4$ turns out to be faithful (see \cite{banica08}).

\subsection{No generalization to \texorpdfstring{$S_n^+$}{Sn+} with \texorpdfstring{$n > 4$}{n > 4}}\label{section:generalisation}

\noindent Finally, we consider similar models of the quantum permutation
group $S_n^+$ with $n > 4$ and show that a generalisation of our main result no
longer holds in this setting. In the following, denote with $A$ again 
the universal unital $C^*$-algebra generated by two projections $p$ and $q$.
Further, recall from \Cref{section:matrix-elimination} the definition of 
the sets $\A_m^{\otimes k} \subseteq A^{\otimes k}$ consisting of tensor 
products of alternating products of 
$p$ and $q$ up to length $m$. 
In this notation, we obtain the following result.

\begin{proposition}
Let $n \in \{5, 6\}$ and $k, \ell \in \N$. If $M \in M_n(A^{\otimes k})$
is a magic unitary with entries $M_{ij} \in \Span{\A_\ell^{\otimes k}}$,
then there exists a separating polynomial for $M$.
\end{proposition}
\begin{proof}
As in the proof of our main theorem, we factor the substitution
homomorphism $\varphi_M$ as 
\begin{center}
\begin{tikzcd}
    \npoly{X_n}  \arrow[rd, "\pi"] \arrow[rr, "\varphi_M"] & &
    A^{\otimes{k}}  \\ &
    \npoly{X_n} / I_n \arrow[ru, "\psi"] &
\end{tikzcd}
\end{center}
and consider the restriction $\psi|_{V_m}$ to the spaces 
$V_m \subseteq \npoly{X_n} / I_n$ of residue classes of polynomials up to degree $m$.
By the computational results in \Cref{section:computatial-results}, we have that for 
$n \in \{5, 6\}$ the 
dimension of $V_m$ grows exponential in $m$. However, one shows similar
to \Cref{lemma:image-dimension} that 
$\dim \psi(V_m) \leq {(2m+1)}^{kl}$.
Hence, $\dim V_m > \dim \psi(V_m)$ for large $m$,
which implies $\dim \ker \psi > 0$.
Thus, we find a residue class of a non-trivial separating polynomial in 
$\ker \psi$.
\end{proof}

\noindent Note that we were only able to compute Gröbner bases for 
the magic unitary ideals $I_n$ with $n \leq 6$, which allowed us to 
prove the exponential growth of $\dim V_m$ in these cases. However, 
$\dim V_m$ should grow even faster for larger $n$, such that we expect 
the previous proposition to hold for all $n > 4$.
Further, the previous dimension argument shows that
the quantum permutation group $S_4^+$ is less complex than
the larger quantum permutation groups $S_n^+$ with $n > 4$.
Compare this to~\cite{banica98}, where it is shown that the 
dual of $S_4^+$ is amenable but it is not amenable for $S_n^+$ with $n > 4$, 
which also shows that $S_4^+$ is somewhat simpler.

\bibliographystyle{alpha}
\bibliography{concrete_model_s4+}

\newcommand{\etalchar}[1]{$^{#1}$}
\begin{thebibliography}{{The}20}

\bibitem[Ban98]{banica98}
T.~Banica.
\newblock Symmetries of a generic coaction.
\newblock {\em Mathematische Annalen}, 314:763–780, 1998.

\bibitem[BB09]{banica09}
T.~Banica and J.~Bichon.
\newblock Quantum groups acting on 4 points.
\newblock {\em Journal für die Reine und Angewandte Mathematik}, 626:75--114, 2009.

\bibitem[BB15]{banica15}
T.~Banica and J.~Bichon.
\newblock Random walk questions for linear quantum groups.
\newblock {\em International Mathematics Research Notices}, 2015(24):13406–13436, 2015.

\bibitem[BBS09]{banica09b}
T.~Banica, J.~Bichon, and J.~Schlenker.
\newblock Representations of quantum permutation algebras.
\newblock {\em Journal of Functional Analysis}, 257(9):2864--2910, 2009.

\bibitem[BC08]{banica08}
T.~Banica and B.~Collins.
\newblock Integration over the pauli quantum group.
\newblock {\em Journal of Geometry and Physics}, 58(8):942--961, 2008.

\bibitem[BCF20]{brannan}
M.~Brannan, A.~Chirvasitu, and A.~Freslon.
\newblock Topological generation and matrix models for quantum reflection groups.
\newblock {\em Advances in Mathematics}, 363:106982, 2020.

\bibitem[Bla05]{blackadar06}
B.~Blackadar.
\newblock {\em Operator algebras. Theory of {C}*-algebras and von {N}eumann algebras}.
\newblock Springer, 2005.

\bibitem[BM07]{banica07}
T.~Banica and S.~Moroianu.
\newblock On the structure of quantum permutation groups.
\newblock {\em Proceedings of the American Mathematical Society}, 135(1):21--29, 2007.

\bibitem[BN17]{banica17}
T.~Banica and I.~Nechita.
\newblock Flat matrix models for quantum permutation groups.
\newblock {\em Advances in Applied Mathematics}, 83:24--46, 2017.

\bibitem[CK16]{GBNP}
A.~Cohen and J.~Knopper.
\newblock {GBNP}, computing {G}röbner bases of noncommutative polynomials, {V}ersion 1.0.3.
\newblock \href {https://www.gap-system.org/Packages/gbnp.html} {\texttt{https://\allowbreak www.gap-\allowbreak system.org/\allowbreak Packages/\allowbreak gbnp.html}}, 2016.

\bibitem[GAP20]{GAP4}
The GAP~Group.
\newblock {\em {GAP -- Groups, Algorithms, and Programming, Version 4.11.0}}, 2020.
\newblock {\href {https://www.gap-system.org} {\texttt{https://\allowbreak www.gap-\allowbreak system.org}}}.

\bibitem[HMU06]{hopcroft}
J.~E. Hopcroft, R.~Motwani, and J.~D. Ullman.
\newblock {\em Introduction to Automata Theory, Languages, and Computation (3rd Edition)}.
\newblock Addison-Wesley Longman Publishing Co., Inc., 2006.

\bibitem[JW20]{jung}
S.~Jung and M.~Weber.
\newblock Models of quantum permutations.
\newblock {\em Journal of Functional Analysis}, 279(2):108516, 2020.

\bibitem[LES{\etalchar{+}}22]{levandovskyy}
V.~Levandovskyy, C.~Eder, A.~Steenpass, S.~Schmidt, J.~Schanz, and M.~Weber.
\newblock Existence of quantum symmetries for graphs on up to seven vertices: A computer based approach.
\newblock {\em ISSAC '22: Proceedings of the 2022 International Symposium on Symbolic and Algebraic Computation}, pages 311--318, 2022.

\bibitem[LVW20]{li}
X.~Li, C.~Voigt, and M.~Weber.
\newblock {ISem24}: {C}*-algebras and dynamics, lecture notes.
\newblock \href{https://www.math.uni-sb.de/ag/speicher/weber/ISem24/ISem24LectureNotes.pdf}{\texttt{https://www.math.uni-sb.de/ag/speicher/weber/ISem24/ISem24LectureNotes.pdf}}, 2020.

\bibitem[Mor94]{mora}
T.~Mora.
\newblock An introduction to commutative and noncommutative {G}röbner bases.
\newblock {\em Theoretical Computer Science}, 134(1):131--173, 1994.

\bibitem[NT13]{neshveyev13}
S.~Neshveyev and L.~Tuset.
\newblock {\em Compact Quantum Groups and Their Representation Categories}.
\newblock Cours Spécialisés. Société Mathématique de France, 2013.

\bibitem[{The}20]{sagemath}
{The Sage Developers}.
\newblock {\em {S}ageMath, the {S}age {M}athematics {S}oftware {S}ystem ({V}ersion 9.1)}, 2020.
\newblock \href {https://www.sagemath.org} {\texttt{https://\allowbreak www.sagemath.org}}.

\bibitem[Tim08]{timmermann08}
T.~Timmermann.
\newblock {\em An Invitation to Quantum Groups and Duality: From {H}opf Algebras to Multiplicative Unitaries and Beyond}.
\newblock EMS textbooks in mathematics. European Mathematical Society, 2008.

\bibitem[Ufn91]{ufnarovski}
V.~Ufnarovskiĭ.
\newblock On the use of graphs for computing a basis, growth and {H}ilbert series of associative algebras.
\newblock {\em Mathematics of The USSR-Sbornik}, 68(2):417--428, 1991.

\bibitem[Wan98]{wang}
S.~Wang.
\newblock Quantum symmetry groups of finite spaces.
\newblock {\em Communications in Mathematical Physics}, 195:195--211, 1998.

\bibitem[Web23]{weber}
M.~Weber.
\newblock Quantum permutation matrices.
\newblock {\em Complex Analysis and Operator Theory}, 17:37, 2023.

\bibitem[Wor87]{woronowicz}
S.~Woronowicz.
\newblock Compact matrix pseudogroups.
\newblock {\em Communications in Mathematical Physics}, 111:613--665, 1987.

\end{thebibliography}
\newpage
\appendix
\section{Finite automaton}\label{appendix:automaton}

\begin{figure}[H]
\newcommand{\opac}{0.35}
\makebox[\textwidth][c]{
\begin{tikzpicture}[scale=0.92, auto, initial text=, >=latex]
\node[state, accepting, initial] (v0) at (0, 1.5) {$0$};
\node[state, accepting] (v5) at (5, -1.5) {$5$};
\node[state, accepting] (v1) at (2, -1.5) {$1$};
\node[state, accepting] (v2) at (-2, -1.5) {$2$};
\node[state, accepting] (v4) at (-5, -1.5) {$4$};
\node[state, accepting] (v3) at (6.5, -5) {$3$};
\node[state, accepting] (v10) at (5.5, -7) {$10$};
\node[state, accepting] (v8) at (-6.5, -5) {$8$};
\node[state, accepting] (v15) at (-5.5, -7) {$15$};
\node[state, accepting] (v6) at (1, -5) {$6$};
\node[state, accepting] (v12) at (4, -5) {$12$};
\node[state, accepting] (v7) at (-1, -5) {$7$};
\node[state, accepting] (v13) at (-4, -5) {$13$};
\node[state, accepting] (v11) at (1, -7.5) {$11$};
\node[state, accepting] (v14) at (-1, -7.5) {$14$};
\node[state, accepting] (v9) at (0, -10) {$9$};
\node[state, accepting] (v16) at (0, -12) {$16$};
\path[->, sloped] (v0) edge node {\footnotesize $x_{11}$} (v1);
\path[->, sloped] (v0) edge node {\footnotesize $x_{12}$} (v2);
\path[->, sloped, opacity=\opac] (v0) edge[bend left=10] node {\footnotesize $x_{13}$} (v3);
\path[->, sloped] (v0) edge[bend right=15] node {\footnotesize $x_{21}$} (v4);
\path[->, sloped] (v0) edge[bend left=15] node {\footnotesize $x_{22}$} (v5);
\path[->, sloped] (v0) edge node {\footnotesize $x_{23}$} (v6);
\path[->, sloped] (v0) edge node {\footnotesize $x_{31}$} (v7);
\path[->, sloped, opacity=\opac] (v0) edge[bend right=10] node {\footnotesize $x_{32}$} (v8);
\path[->, sloped, opacity=\opac] (v0) edge[below] node {\footnotesize $x_{33}$} (v9);
\path[->, sloped] (v1) edge[bend left] node {\footnotesize $x_{22}$} (v5);
\path[->, sloped] (v1) edge node {\footnotesize $x_{23}$} (v6);
\path[->, sloped, opacity=\opac] (v1) edge node {\footnotesize $x_{32}$} (v8);
\path[->, sloped, opacity=\opac] (v1) edge[bend left] node {\footnotesize $x_{33}$} (v9);
\path[->, sloped] (v2) edge[bend left, below] node {\footnotesize $x_{21}$} (v4);
\path[->, sloped] (v2) edge node {\footnotesize $x_{23}$} (v6);
\path[->, sloped] (v2) edge node {\footnotesize $x_{31}$} (v7);
\path[->, sloped, opacity=\opac] (v2) edge[bend right] node {\footnotesize $x_{33}$} (v9);
\path[->, sloped] (v3) edge[bend left, below] node {\footnotesize $x_{21}$} (v10);
\path[->, sloped] (v3) edge node {\footnotesize $x_{31}$} (v11);
\path[->, sloped] (v4) edge[bend left] node {\footnotesize $x_{12}$} (v2);
\path[->, sloped, opacity=\opac] (v4) edge node {\footnotesize $x_{13}$} (v3);
\path[->, sloped] (v4) edge node {\footnotesize $x_{32}$} (v8);
\path[->, sloped] (v4) edge[bend right=80, looseness=1.6] node {\footnotesize $x_{33}$} (v9);
\path[->, sloped] (v5) edge[bend left] node {\footnotesize $x_{11}$} (v1);
\path[->, sloped] (v5) edge node {\footnotesize $x_{13}$} (v3);
\path[->, sloped] (v5) edge node {\footnotesize $x_{31}$} (v7);
\path[->, sloped] (v5) edge[bend left=80, looseness=1.6] node {$x_{33}$} (v9);
\path[->, sloped] (v6) edge node {\footnotesize $x_{31}$} (v11);
\path[->, sloped] (v6) edge[bend left] node {\footnotesize $x_{11}$} (v12);
\path[->, sloped] (v7) edge[bend left] node {\footnotesize $x_{12}$} (v13);
\path[->, sloped] (v7) edge node {\footnotesize $x_{13}$} (v14);
\path[->, sloped] (v8) edge[below] node {\footnotesize $x_{13}$} (v14);
\path[->, sloped] (v8) edge[bend left, below] node {\footnotesize $x_{11}$} (v15);
\path[->, sloped] (v9) edge[bend left, below] node {\footnotesize $x_{11}$} (v16);
\path[->, sloped] (v10) edge[bend left, below] node {\footnotesize $x_{13}$} (v3);
\path[->, sloped] (v10) edge[bend left=20] node {\footnotesize $x_{33}$} (v9);
\path[->, sloped] (v11) edge[bend left, below] node {\footnotesize $x_{13}$} (v14);
\path[->, sloped] (v12) edge[bend left] node {\footnotesize $x_{23}$} (v6);
\path[->, sloped, opacity=\opac] (v12) edge[bend left=20] node {\footnotesize $x_{33}$} (v9);
\path[->, sloped] (v13) edge[bend left] node {\footnotesize $x_{31}$} (v7);
\path[->, sloped, opacity=\opac] (v13) edge[bend right=20] node {\footnotesize $x_{33}$} (v9);
\path[->, sloped] (v14) edge[bend left] node {\footnotesize $x_{31}$} (v11);
\path[->, sloped] (v15) edge[bend left, below] node {\footnotesize $x_{32}$} (v8);
\path[->, sloped] (v15) edge[bend right=20] node {\footnotesize $x_{33}$} (v9);
\path[->, sloped] (v16) [bend left, below] edge node {\footnotesize $x_{33}$} (v9);
\end{tikzpicture}
}
\caption{The finite automaton constructed in \Cref{section:quotient-basis} for the magic unitary ideal $I_4$.
It has $17$ states, where state $0$ is the initial state and every state is final. The grey edges are only coloured differently
for better visualisation.}\label{figure:automaton-I4}
\end{figure}
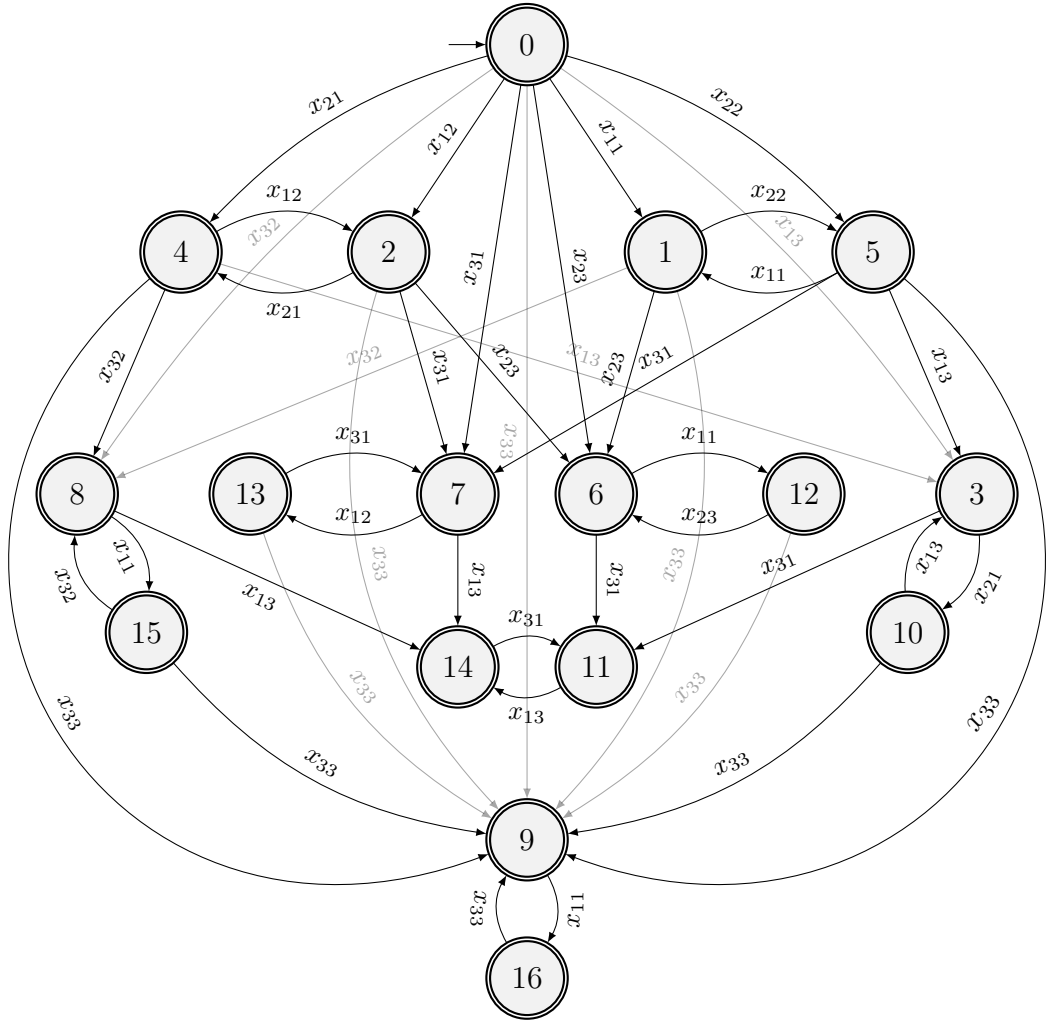

\newpage
\section{Characters from Proposition~\ref{proposition:spectrum}}\label{appendix:characters}

\begin{figure}[H]
\begin{tabular}{c|cccccc}
$S_4$ & $p_1$ & $q_1$ & $p_2$ & $q_2$ & $p_3$ & $q_3$ \\
\hline
() & 0 & 1 & 0 & 1 & 0 & 1 \\
(13)(24) & 0 & 1 & 0 & 1 & 1 & 0 \\
(14)(23) & 0 & 1 & 1 & 0 & 0 & 1 \\
(12)(34) & 0 & 1 & 1 & 0 & 1 & 0 \\
(234) & 0 & 0 & 1 & 1 & 0 & 1 \\
(132) & 0 & 0 & 1 & 1 & 1 & 0 \\
(143) & 0 & 0 & 0 & 0 & 0 & 1 \\
(124) & 0 & 1 & 0 & 0 & 0 & 0 \\
(243) & 0 & 0 & 0 & 1 & 0 & 0 \\
(134) & 0 & 0 & 0 & 1 & 1 & 1 \\
(142) & 0 & 0 & 1 & 0 & 1 & 1 \\
(123) & 0 & 0 & 1 & 0 & 0 & 0 \\
(34) & 0 & 0 & 0 & 1 & 0 & 1 \\
(1324) & 0 & 0 & 0 & 1 & 1 & 0 \\
(1423) & 0 & 0 & 1 & 0 & 0 & 1 \\
(12) & 0 & 0 & 1 & 0 & 1 & 0 \\
(23) & 0 & 0 & 1 & 1 & 0 & 0 \\
(1342) & 0 & 0 & 1 & 1 & 1 & 1 \\
(14) & 0 & 1 & 0 & 0 & 0 & 1 \\
(1243) & 0 & 0 & 0 & 0 & 0 & 0 \\
(24) & 0 & 1 & 0 & 1 & 0 & 0 \\
(13) & 0 & 1 & 0 & 1 & 1 & 1 \\
(1432) & 0 & 1 & 1 & 0 & 1 & 1 \\
(1234) & 0 & 1 & 1 & 0 & 0 & 0
\end{tabular}
\end{figure}

\end{document}